\documentclass[10pt]{amsart}

\usepackage[utf8]{inputenc}
\usepackage{amsmath}
\usepackage{amssymb}
\usepackage{amsfonts}
\usepackage{amsthm}
\usepackage{thmtools}
\usepackage{enumerate}
\usepackage[top=3cm, bottom=3cm, left=3cm, right=3cm]{geometry}
\usepackage{graphicx}
\usepackage[all]{xy}
\usepackage{titlesec}
\usepackage{hyperref}
\usepackage{rotating}
\usepackage{xcolor}
\usepackage{setspace}
\usepackage[T1]{fontenc}
\usepackage{verbatim}

\titleformat{\section}
{\filcenter\large\bf} 
{\thesection.\ } 
{1pt} 
{} %

\titleformat{\subsection}[hang]
{\filcenter\bf}
{\thesubsection.}
{1pt}
{}

\declaretheoremstyle[bodyfont=\normalfont]{normalbody}

\declaretheorem[numberwithin=section,name=Theorem]{theorem}
\declaretheorem[sibling=theorem,style=normalbody,name=Definition]{definition}
\declaretheorem[sibling=theorem,name=Corollary]{corollary}
\declaretheorem[sibling=theorem,name=Lemma]{lemma}
\declaretheorem[sibling=theorem,name=Proposition]{proposition}

\declaretheorem[sibling=theorem,style=normalbody,name=Remark]{remark}

\newcommand{\Z}{\mathbb{Z}}
\newcommand{\N}{\mathbb{N}}
\newcommand{\Q}{\mathbb{Q}}

\newcommand{\F}{\mathbb{F}}
\newcommand{\p}{\mathfrak{p}}
\renewcommand{\P}{\mathfrak{P}}

\newcommand{\ent}{\mathcal{O}}

\newcommand{\cmod}[3]{#1\equiv#2\pmod{#3}}

\newcommand{\spot}[1]{\mathfrak{#1}}

\newcommand{\Hom}{\operatorname{Hom}}
\newcommand{\Aut}{\operatorname{Aut}}
\newcommand{\Gal}{\operatorname{Gal}}
\newcommand{\Br}{\operatorname{Br}}

\newcommand{\nr}{\operatorname{nr}}
\newcommand{\cyc}{\operatorname{cyc}}

\newcommand{\Res}{\operatorname{Res}}

\newcommand{\inv}{\operatorname{inv}}
\newcommand{\Tate}{\hat{H}}
\newcommand{\Bad}{\operatorname{Bad}}

\DeclareFontFamily{U}{wncy}{}
\DeclareFontShape{U}{wncy}{m}{n}{<->wncyr10}{}
\DeclareSymbolFont{mcy}{U}{wncy}{m}{n}
\DeclareMathSymbol{\Sh}{\mathord}{mcy}{"58}

\DeclareFontFamily{U}{wncy}{}
\DeclareFontShape{U}{wncy}{m}{n}{<->wncyr10}{}
\DeclareSymbolFont{mcy}{U}{wncy}{m}{n}
\DeclareMathSymbol{\Ch}{\mathord}{mcy}{"51}

\title{Bad places for the approximation property for finite groups}

\author{Felipe Rivera-Mesas}
\address{Departamento de Matemáticas, Facultad de Ciencias, Universidad de Chile}
\email{felipe.rivera.m@ug.uchile.cl}
\thanks{ 
		\mbox{\hspace{1.3em}} {\sl MSC codes (2020):} 11S25, 11R34. \\
		\mbox{\hspace{1.3em}} This work was partially supported by the Beca de Magíster Nacional 2020 via Anid.}

\begin{document}

\maketitle

\begin{abstract}
	Given a number field $k$ and a finite $k$-group $G$, the Tame Approximation Problem for $G$ asks whether the restriction map $H^1(k,G)\to\prod_{v\in\Sigma}H^1(k_v,G)$ is surjective for every finite set of places $\Sigma\subseteq\Omega_k$ disjoint from $\Bad_G$, where $\Bad_G$ is the finite set of places that either divides the order of $G$ or ramifies in the minimal extension splitting $G$. In this paper we prove that the set $\Bad_G$ is ``sharp''. To achieve this we prove that there are finite abelian $k$-groups $A$ where the map $H^1(k,A)\to\prod_{v\in\Sigma_0}H^1(k_v,A)$ is not surjective in a set $\Sigma_0\subseteq\Bad_A$ with particular properties, namely $\Sigma_0$ is the set of places that do not divide the order of $A$ and ramify in the minimal extension splitting $A$.
\end{abstract}

%
%
%
%

\section{Introduction}

	Let $k$ be a number field and $G$ be a finite $k$-group i.e.\ a finite group with a continuous action of the profinite group $\Gal(k):=\Gal(\overline{k}/k)$ which is compatible with the group structure of $G$. A finite $k$-group $G$ has approximation in a finite set of places $\Sigma\subseteq\Omega_k$ if the restriction map
		\begin{equation} \label{flecha aprox}
			H^1(k,G) \to \prod_{v\in\Sigma} H^1(k_v,G),
		\end{equation}
	is surjective (cf.\ Definition \ref{def aprox debil}). Under this definition, we say that a finite $k$-group has \emph{the approximation property} if it has approximation in every finite set of places of $k$. This property is in general too strong and it is not satisfied, even for abelian groups. 
	
	\vspace{1em}
	
	To illustrate the fact that studying the approximation property for a finite $k$-group $G$ is a hard problem, consider what happens when $G$ is constant (i.e.\ $\Gal(k)$ acts on $G$ trivially). In this case the problem of determining whether $G$ has approximation in $\Sigma$ becomes \emph{the Grunwald Problem} in $\Sigma$, i.e.\ determining whether the restriction map
		\[ \Hom(\Gal(k),G) \to \prod_{v\in\Sigma} \Hom(\Gal(k_v),G)/\sim, \]
	is surjective, where $\sim$ is the equivalence relation defined by conjugation and $\Hom$ denotes the set of continuous homomorphism. Affirmative answers to the Grunwald Problem for every finite set of places $\Sigma$ are known in many cases. For example: abelian groups of odd order over every number field, by the Grunwald-Wang Theorem (cf.\ \cite{Grunwald1933} and \cite{Wang1950}); and solvable groups of order prime to the number of roots of unity in $k$, by Neukirch's Theorem (cf.\ \cite{Neukirch1979}). Note that an affirmative answer to the Grunwald problem for every $\Sigma$ implies an affirmative answer to the inverse Galois problem. In fact, an affirmative answer to the Grunwald Problem for infinitely many $\Sigma$'s already implies this.
	
	\vspace{1em}
	
	A weaker property than the previously defined is the following. A finite $k$-group has approximation \emph{away from a set of places} $T\subseteq\Omega_k$ if it has approximation in every finite set of places disjoint from $T$. A finite $k$-group that satisfies this for a certain finite $T$ is said to have \emph{weak weak approximation}. So, if a constant $k$-group $G$ has approximation away from a finite set of places of $k$, then it is a Galois group over $k$. The hypothesis of this implication can be relaxed to an even weaker notion (cf.\ \cite[$\S$4, Proposition 1]{Harari2007}).

	\vspace{1em}
	
	Studying these approximation properties for finite $k$-groups is equivalent to studying certain approximation properties for certain homogeneous spaces (cf.\ \cite[Proposition 2.4]{Grunwald}). Now we will briefly recall some implications of this equivalence.
	
	\vspace{1em}
	
	Let $X$ be a (smooth and geometrically integral) $k$-variety such that $X(k)\neq\varnothing$. We say that $X$ has \emph{weak approximation away from} $T\subseteq\Omega_k$ if the image of $X(k)$ by the diagonal embedding in $\prod_{v\in S}X(k_v)$ is dense for every finite set $S\subseteq\Omega_k$ disjoint from $T$.
	
	Now we briefly recall the Brauer-Manin obstruction to weak approximation for $k$-varieties (cf.\ \cite[Chapter 5]{torsors}). Let $\Br_{\nr}X$ be the unramified Brauer group of $X$ (cf. \cite[p. 97, \emph{Brauer groups}]{torsors}). For each place $v\in\Omega_k$ we denote by
		\[ \inv_v:\Br k_v\to\Q/\Z, \] 
	 the Hasse invariant (cf.\ \cite[Proposition 8.4]{HarariCG}). We denote the product $\prod_{v\in\Omega_k}X(k_v)$ by $X(k_\Omega)$. We define $X(k_\Omega)^{\Br_{\nr}}$ as the set of points $(P_v)_{v\in\Omega_k}\in X(k_\Omega)$ such that
	 	\[ \sum_{v\in\Omega_k}\inv_v(\alpha(P_v))=0,\text{ for every }\alpha\in\Br_{\nr}X, \]
	where $\alpha(P_v)\in\Br k_v$ denotes the evaluation of $\alpha$ at the point $P_v$ and the sum, which a priori is infinite, is in fact finite for elements $\alpha\in\Br_{\nr}X$ (cf.\ \cite[\S 5.2]{torsors}). Then we have the following inclusions:
		\[ \overline{X(k)}\subseteq X(k_\Omega)^{\Br_{\nr}}\subseteq X(k_\Omega), \]
	where $\overline{X(k)}$ is the topological closure of the image in $X(k_\Omega)$ of $X(k)$ via the diagonal embedding. If $X(k_\Omega)^{\Br_{\nr}}\neq X(k_\Omega)$ we say that there is a Brauer-Manin obstruction to weak approximation.

	\vspace{1em}

	A conjecture by Colliot-Thélène (cf.\ \cite[Introduction]{CT2003}) says that the Brauer-Manin obstruction to weak approximation should be the only obstruction for homogeneous spaces $X$ under a connected linear group, that is $\overline{X(k)}=X(k_\Omega)^{\text{Br}_\text{nr}}$. Now, given a finite $k$-group we can regard it as a finite algebraic $k$-group which can be embedded into $\mathrm{SL}_n$ for some $n\in\N$. Then, if we consider the quotient space $X:=\text{SL}_n/G$, we get a homogeneous space under the connected group $\text{SL}_n$. Furthermore, since $\Br_{\nr}X/\Br k$ is finite (cf. \cite[\S 13.1]{CTSBrauer} and \cite[Theorem 4.4.2]{CTSBrauer}), the equality $\overline{X(k)}=X(k_\Omega)^{\text{Br}_\text{nr}}$ for the $k$-variety $\mathrm{SL}_n/G$ implies that the finite $k$-group $G$ has weak weak approximation. In other words, Colliot-Thélène's conjecture would imply that, for every finite $k$-group $G$, there is a finite set of places $T(G,k)$ where $G$ has approximation away from it. This conjecture has been proved in many cases, e.g. for iterated semidirect products of finite abelian groups by Harari (cf.\ \cite[Théorème 1]{Harari2007}) and hypersolvable groups by Harpaz and Wittenberg (cf.\ \cite[Théorème 6.6]{HW2020}).
	
	Thus for a $k$-group $G$ as above we have that there is a finite set of places $T\subset\Omega_k$ such that $G$ has approximation away from $T$ (cf.\ Definition \ref{def aprox debil}). Unfortunately, this proof of weak weak approximation, which use the Brauer-Manin obstruction, does not allow us to obtain an explicit description of the set $T$.
	
	\vspace{1em}
	
	In \cite{Grunwald} Demarche, Lucchini Arteche and Neftin defined, for every finite $k$-group $G$, the finite set of places $\Bad_G\subseteq\Omega_k$ as the union of the set $\Bad_G^d$ of places of $k$ that divide the order of $G$ and the set $\Bad_G^r$ of places of $k$ that ramify in the minimal extension splitting $G$ (cf.\ \cite[Definition 2.1]{Grunwald}). Following the result by Harari, they proved that if $G$ is an iterated semidirect product of finite abelian $k$-groups, then $G$ has approximation away from $\Bad_G$ (cf.\ \cite[Theorem 1.1]{Grunwald}). Moreover, they asked whether every finite $k$-group has approximation away from $\Bad_G$. They called this question the \emph{Tame Approximation Problem} (cf.\ \cite[\S 1.2]{Grunwald}). Note that this result cannot be considered as an improvement of the result given by Harari, but rather as a complement of it. Indeed, despite the fact that \cite[Theorem 1.1]{Grunwald} describes explicitly a finite set of places where $G$ has approximation away from it, it does not explain the behavior of the map \eqref{flecha aprox} in subsets of $\Bad_G$. 
	
	Recently, in \cite{Lucchini2019} Lucchini Arteche gave a link between the results above. Given a finite $k$-group $G$, and under the assumption of Colliot-Thélène's conjecture on the Brauer-Manin obstruction, we have an affirmative answer to the Tame Approximation Problem for $G$ (cf.\ \cite[Corollary 6.3]{Lucchini2019}). This result is obtained by describing explicitly the group $\Br_{\nr}X$ with $X=\mathrm{SL}_n/G$. In this way, we get, for every finite $k$-group, an explicit and general description of the set of places of $k$ where the weak weak approximation property is satisfied.
	
	\vspace{1em}
	

	Therefore, studying the set $\Bad_G$ is interesting itself. In particular, it is interesting to know how much the subsets $\Bad_G^r$ and $\Bad_G^d$ affect approximation. By \cite[Theorem 5.1]{Grunwald}, there are examples of finite $k$-groups which do not have approximation in $\Sigma\subseteq\Bad_G^d$ (cf.\ Definition \ref{def aprox debil}). These examples satisfy $\Bad_G^r=\varnothing$ and are necessarily non abelian. Even more, here below, we will note that such examples cannot be obtained among abelian finite $k$-groups. Now what happens with $\Bad_G^r$? We will prove that there are finite $k$-groups that do not have approximation in $\Sigma\subseteq\Bad_G^r\smallsetminus\Bad_G^d$. These examples will be abelian so, by standard arguments in group cohomology, we will concentrate on $\ell$-primary torsion abelian $k$-groups, with $\ell$ a prime.
	
	More precisely, we will exhibit explicit examples of finite abelian $k$-groups $A$ and sets $\Sigma\subseteq\Bad_A^r\smallsetminus\Bad_A^d$ such that $A$ does not have approximation in $\Sigma$ but it does in $\Sigma\smallsetminus\{v\}$ for every $v\in\Sigma$. These examples are quite general. Indeed, firstly, for every $\ell$ there is such an example where $A$ is an $\ell$-group (cf.\ Theorem \ref{no aproximacion en l-grupos}); and secondly, for every number field $k$ and every place $\p\in\Omega_k$ not dividing 2, there is such an example where $\p\in\Sigma$ and $A$ has approximation in $\Sigma\smallsetminus\{\p\}$ (cf.\ Theorem \ref{no aproximacion en cuerpos}). Thus, these examples prove that the set $\Bad_G$ is ``sharp'' in the following sense: if $f$ assigns to each finite $k$-group $G$ a finite set of places $f(G)\subseteq\Omega_k$ such that $G$ has approximation away from $f(G)$, then $f(G)$ cannot have less elements than $\Bad_G$ for every finite $k$-group $G$. That is, the subsets $\Bad_G^r$ and $\Bad_G^d$ cannot be ignored in general.
	
%
%
	
\section*{Acknowledgement}

	I would like to thank my supervisor Giancarlo Lucchini Arteche for suggesting this problem and for his helpful comments about the content and redaction of this article. Besides, I would like to thank professor Cyril Demarche for reviewing and for his useful comments about this work.

%
%

\section{Conventions and notations}

	All throughout this text $k$ denotes a number field and $\Omega_k$ the set of places of $k$. For each $v\in \Omega_k$ we denote by $k_v$ the completion of $k$ at $v$ and, if $v$ is a non archimedean place, $\kappa(v)$ denotes the residue field at $v$. Besides, for every field $K$, we denote by $\Gal(K)$ its absolute Galois group.
	If $L/k$ is a Galois extension with group $G$, for each $v\in\Omega_k$ we define its decomposition subgroup $G_v$ in $L/k$ as the decomposition subgroup $G(w/v)$ in $L/k$ of some place $w\in\Omega_L$ above $v$. In this case $\Gal(L_w/k_v) = G_v$. Note that this notion depends on $w$ but this is not a issue in the first cohomology since if $w$ and $w'$ are places over $v$ then $G(w/v)=g^{-1}G(w'/v)g$ for some $g\in G$ and therefore the cohomology sets $H^1(G(w/v),A)$ and $H^1(G(w'/v),A)$ are canonically isomorphic for every $G$-group $A$.

	We say that $A$ is a finite $k$-group if $A$ is a group with a continuous action of $\Gal(k)$ compatible with its group structure. When $A$ is abelian we say $A$ is an abelian $k$-group. Moreover, for each finite abelian $k$-group $A$ we denote by $\hat{A}$ its Cartier dual, i.e.\ the $k$-group $\Hom(A,\overline{k}^*)$ with the continuous action $\sigma\cdot f(x)=\sigma\cdot f(\sigma^{-1}\cdot x)$. An algebraic extension $L/k$ is said to split $A$ if $\Gal(L)$ acts trivially on $A$. Likewise $A$ is said to be constant if $\Gal(k)$ acts on $A$ trivially. We note that for every finite $k$-group $A$ there exists a unique minimal finite Galois extension $L/k$ splitting $A$. Indeed, this extension corresponds to the kernel of the morphism $\Gal(k)\to\Aut(A)$ induced by the action of $\Gal(k)$ on $A$.	

\section{Preliminaries}

	In this section, we recall some definitions related to the approximation property for $k$-groups and some known results in this topic. \\
	
	We recall the notion of approximation in $k$-groups that was mentioned in the introduction.

	\begin{definition} \label{def aprox debil}
		Let $G$ be a $k$-group. We say \emph{$G$ has approximation in a set $\Sigma\subseteq\Omega_k$} if the restriction map
			\[ H^1(k,G) \to \prod_{v\in\Sigma} H^1(k_v,G), \]
		is surjective. Likewise, we say \emph{$G$ has approximation away from $\Sigma$}, if $G$ has approximation for every finite set of places of $k$ that is disjoint from $\Sigma$. For convenience, we assume that $G$ has approximation in $\Sigma=\varnothing$.
	\end{definition}
	
	An important object to study the approximation property for $k$-groups is the following.

	\begin{definition}
		Let $A$ be a $k$-group and $\Sigma$ be a set of places of $k$. We denote by $\Sh^1_\Sigma(k,A)$ the kernel of the restriction map
			\[ H^1(k,A) \to \prod_{v\notin\Sigma} H^1(k_v,A). \]
		We write $\Sh^1(k,A):=\Sh^1_{\varnothing}(k,A)$ and we define $\Sh^1_\omega(k,A)$ as the union of the groups $\Sh^1_\Sigma(k,A)$ for every finite $\Sigma$. For a Galois extension $L/k$ splitting $A$ with Galois group $G$, we denote by $\Sh^1_\Sigma(L/k,A)$ the kernel of the restriction map
			\[ H^1(G,A) \to \prod_{v\notin\Sigma} H^1(G_v,A), \]
		and $\Sh^1(L/k,A):=\Sh^1_{\varnothing}(L/k,A)$.
	\end{definition}
	
	The following proposition, which corresponds to \cite[Proposition 2.5]{Grunwald}, gives us a link between the previously defined concepts.

	\begin{proposition} \label{aproximacion y sha}
		Let $A$ be a abelian $k$-group and $\Sigma\subseteq\Omega_k$ be a finite set of places. Then there is a pairing
			\[ \prod_{v\in\Sigma}H^1(k_v,A)\times\Sh_\Sigma^1(k,\hat{A})\to\Q/\Z, \]
		such that its right kernel is $\Sh^1(k,A)$ and its left kernel is the image of the restriction map
			\[ H^1(k,A) \to \prod_{v\in\Sigma} H^1(k_v,A). \]
		In particular, $A$ has approximation in $\Sigma$ if and only if $\Sh_\Sigma^1(k,\hat{A})=\Sh^1(k,\hat{A})$.
	\end{proposition}
	
	Let $A$ be a finite abelian $k$-group and $L/k$ be an extension splitting $A$. By \cite[Lemme 1.1 (iii)]{Sansuc1981} and Proposition \ref{aproximacion y sha}, we note that if $\Sigma$ is a finite set of places of $k$ where $A$ does not have approximation, then that set must contain at least one place whose decomposition group in $L/k$ is not cyclic. Even more, one can assume that $\Sigma$ only contains places whose decomposition group in $L/k$ is not cyclic. In particular, these places must ramify in $L/k$. Therefore, by Proposition \ref{aproximacion y sha}, we get that $A$ has approximation in every set $\Sigma\subseteq\Omega_k\smallsetminus\Bad_A^r$, in particular for $\Sigma\subseteq\Bad_A^d\smallsetminus\Bad_A^r$. Therefore, unlike \cite[Theorem 5.1]{Grunwald}, we cannot find a finite $k$-group $A$ and a finite set of places $\Sigma\subseteq\Bad_A^d\smallsetminus\Bad_A^r$ where $A$ does not have approximation and $A$ is abelian. 
	
	In addition, by Proposition \ref{aproximacion y sha}, we note that archimedean places play no role in the lack of surjectivity of \eqref{flecha aprox} for finite abelian $k$-groups. Indeed, these places have cyclic decomposition groups, hence $\Sh^1_{\Sigma\smallsetminus\Omega_\infty}(k,A)=\Sh^1_{\Sigma}(k,A)$ for every finite set of places $\Sigma$, where $\Omega_\infty$ is the set of archimedean places of $k$. Therefore, throughout this text we will assume that every finite set of places $\Sigma\subseteq\Omega_k$ only has non archimedean places.

	Now, we give a definition that will allow us to study the approximation property of an abelian $k$-group from a purely algebraic point of view.

	\begin{definition} 
		Let $G$ be a finite group. For a $G$-module $A$ we define $\Sh_{\cyc}^1(G,A)$ as the kernel of the restriction map
			 \[ H^1(G,A) \to \prod_{g\in G} H^1(\langle g\rangle,A). \]
	\end{definition}
	
	By \cite[Lemme 1.2]{Sansuc1981} we get the equalities 
		\begin{equation} \label{sha cyc omega}
			\Sh^1_{\cyc}(L/k,A)=\Sh_\omega^1(k,A)=\Sh_{\Sigma_0(k,A)}^1(k,A),
		\end{equation}
	for every finite abelian $k$-group $A$, where $L/k$ is the minimal extension splitting $A$ and the set $\Sigma_0(k,A)$ is defined as follows.
	
	\begin{definition}
		Let $A$ be a finite abelian $k$-group. We define $\Sigma_0(k,A)$ as the set of places of $k$ whose decomposition group in the minimal extension splitting $A$ is not cyclic. 
	\end{definition}
	
	Proposition \ref{aproximacion y sha} tells us then the following: if the finite abelian $k$-group $A$ does not have approximation in the finite set $\Sigma\subseteq\Omega_k$ then $\Sigma\cap\Sigma_0(k,\hat{A})\neq\varnothing$. In other words, when $A$ is a finite abelian $k$-group the set of ``bad places'' of $A$ is $\Sigma_0(k,\hat{A})$. This is consistent with the following result.
	
	\begin{proposition} \label{sigma0}
		The set $\Sigma_0(k,\hat{A})$ is contained in $\Bad_A$
	\end{proposition}
	\begin{proof}
		If $A$ is a finite abelian $k$-group and $L/k$ is the minimal extension splitting $A$, then the minimal extension splitting $\hat{A}$ is contained in $L(\zeta_e)/k$, where $e$ is the exponent of $A$. Therefore, if a place $v\in\Omega_k$ ramifies in $L(\zeta_e)/k$, then it ramifies either in $L/k$ or in $L(\zeta_e)/L$, which necessarily implies that either $v$ ramifies in $L/k$ or divides $e$ (hence divides the order of $A$).
	\end{proof}
		
%
%
	
\section{Construction of non trivial quotients $\Sh_\Sigma^1(k,A)/\Sh^1(k,A)$} 

	By Proposition \ref{aproximacion y sha} and the equality \eqref{sha cyc omega}, if a finite abelian $k$-group $A$ does not have approximation in $\Sigma\subseteq\Omega_k$, then $\Sh^1_\Sigma(k,\hat{A})\neq 0$, so that $\Sh^1_{\cyc}(k,\hat{A})\neq 0$. Therefore, we will be interested in constructing finite abelian $k$-groups $A$ for which $\Sh^1_{\cyc}(G,A)$ is non trivial, where $G$ is the Galois group of the minimal extension splitting $A$.
	
	In order to construct these $k$-groups, we use the following lemma.
	

	\begin{lemma} \label{sha omega no trivial}
		Let $G$ be a finite group of order $n$ and exponent $e$. Let $I$ be the augmentation ideal of $\Z/n\Z[G]$. Then $\Sh_{\cyc}^1(G,I)\cong\Z/f\Z$ where $f=n/e$.
	\end{lemma}
	\begin{proof}
		We have the exact sequence $0 \to I \to \Z/n\Z[G] \to \Z/n\Z \to 0$, where $\Z/n\Z[G]$ is cohomologically trivial (cf. \cite[Proposition 1.23]{HarariCG}). Then for all $H\leq G$ we get isomorphisms $H^1(H,I)\cong\Tate^0(H,\Z/n\Z)$, where $\Tate^0$ denotes the Tate cohomology group. We have $\Tate^0(H,\Z/n\Z)\cong\Z/|H|\Z$ for all $H\leq G$. Hence
			\[ \Sh_{\cyc}^1(G,I) \cong \ker \left[ \Z/n\Z \to \prod_{g\in G} \Z/|g|\Z \right] \cong \Z/f\Z, \]
		where $f=n/e$.
	\end{proof}

	To complement the previous result we note that, in order to ensure that an abelian $k$-group $A$ does not have approximation in $\Sigma$, it is sufficient to satisfy the following conditions
		\begin{enumerate}[(i)]
			\item $\Sh^1_\Sigma(k,\hat{A})\neq 0$; \label{i}
			\item $\Sh^1(k,\hat{A})=0$. \label{ii}
		\end{enumerate}
	So, we will be interested in knowing when the second condition is satisfied since we already know how to construct examples where the first condition is satisfied. The following lemma will help us in this direction.

	\begin{lemma} \label{sha omega trivial}
		Let $L/k$ be the minimal extension splitting the finite abelian $k$-group $A$. Suppose that there is a place $v\in\Omega_k$ such that $G_v=\Gal(L/k)$. Then $\Sh^1_{\Sigma}(k,A)=0$ for every finite set of places $\Sigma\subset\Omega_k$ which does not contain $v$.
	\end{lemma}
	\begin{proof}
		Let $v$ be a place of $k$ such that $G_v=G$. Then the restriction map 
			\[ \Res_v:H^1(G,A) \to H^1(G_v,A) \]
		is the identity. Let $\Sigma\subset\Omega_k$ be a finite set of places which does not contain $v$. So, if $\alpha\in\Sh^1_\Sigma(L/k,A)$ then $\Res_v(\alpha)=\alpha=0$. Thus, $\Sh^1_\Sigma(L/k,A)=0$. Now, by \cite[Lemme 1.1 (ii)]{Sansuc1981} we have $\Sh^1_\Sigma(k,A)=\Sh^1_\Sigma(L/k,A)$. Hence, $\Sh^1_\Sigma(k,A)=0$.
	\end{proof}
	
	The following results are obtained immediately from the previous lemma and give us a sufficiency condition to get $\Sh^1(k,A)=0$.

	\begin{corollary} \label{no achicar}
		Let $L/k$ be the minimal extension splitting the finite abelian $k$-group $A$. Assume that $A$ does not have approximation in $\Sigma$. Then $\Sigma$ contains all places whose decomposition group in $L/k$ is $\Gal(L/k)$. In particular, if $v\in\Omega_k$ is a place whose decomposition group in $L/k$ is $\Gal(L/k)$, then $A$ has approximation in $\Sigma_0(k,\hat{A})\smallsetminus\{v\}$.
	\end{corollary}

	\begin{corollary} \label{sha trivial}
		Let $L/k$ be the minimal extension splitting the finite abelian $k$-group $A$. If there exists a place $v\in\Omega_k$ whose decomposition group in $L/k$ is $G:=\Gal(L/k)$, then $\Sh^1(k,A)=0$.
	\end{corollary}

	Let $G$ be a finite group of order $n$ and $L/k$ be a Galois extension with group $G$. The augmentation ideal of $\Z/n\Z[G]$ has a natural structure of abelian $k$-group. Indeed, $\Gal(L/k)$ acts (continuously) on $I$ via the natural action of $G$ on $I$. Extending this action to $\Gal(k)$ we have a structure of abelian $k$-group for $I$ where the minimal extension splitting $I$ is $L/k$ since $\Gal(L/k)$ acts faithfully on $\{g-1|g\in G\}$. In this way, we can ask:
		\begin{center}
			Is there a Galois extension $L/k$ with group $G$ of order $n$ where the abelian $k$-group $I$ and the set of places $\Sigma_0(k,I)$ satisfy the conditions \eqref{i} and \eqref{ii}?
		\end{center}
	
	At this point is important to note the following: if $\Sigma_0(k,I)$ does not contain any place dividing $n$ then $\Sigma_0(k,\hat{A})\subseteq\Bad_A^r\smallsetminus\Bad_A^d$ for $A=\hat{I}$.
	
	The following lemma answers the question above concretely and will guide our constructions.
	
	\begin{lemma} \label{cociente no trivial}
		Let $L/k$ be a Galois extension with Galois group $G$ of order $n$ and exponent $e$ such that $n/e\neq 1$ (for example, if $G$ is a non cyclic abelian group). Let $I$ be the finite abelian $k$-group defined as the augmentation ideal of $\Z/n\Z[G]$. Then, if there exists a place $v\in\Omega_k$ such that $G_v=G$, we have $\Sh_{\Sigma_0(k,I)}^1(k,I)/\Sh^1(k,I)\neq 0$.
	\end{lemma}
	\begin{proof}
		By Lemma \ref{sha omega no trivial} we have $\Sh_{\Sigma_0(k,I)}^1(k,I)=\Sh_{\cyc}^1(G,I)\neq 0$ since $n/e\neq 1$. Now let us suppose that there is a place $v\in\Omega_k$ such that $G_v=G$, then by Corollary \ref{sha trivial} we have $\Sh^1(k,I)=0$. Hence, the quotient $\Sh^1_{\Sigma_0(k,I)}(k,I)/\Sh^1(k,I)$ is non trivial.
	\end{proof}	
	
%
%
%
%

\section{Failure of approximation on abelian $k$-groups}

	Throughout this section we will use latin letters (e.g. $p$) to denote primes in $\Q$ and gothic letters (e.g. $\p,\frak{P}$) to denote places in number fields distinct from $\Q$.

%
%

\subsection{Failure of approximation on $\ell$-groups}

	The aim of this subsection is to find a family of counterexamples to the approximation on $\ell$-groups when $\ell$ is a given prime. This is summarized in the following theorem.

	\begin{theorem} \label{no aproximacion en l-grupos}
		For every pair of primes $\ell$ and $p$ with $\cmod{p}{1}{\ell^n}$, there exist a number field $k$, a finite set of places $\Sigma\subseteq\Omega_k$ and an abelian $k$-group $A$ of order $\ell^a$, with $a=(n+1)(\ell^{n+1}-1)$, such that
			\begin{itemize}
				\item $\Sigma\subseteq\Bad^r_A\smallsetminus\Bad^d_A$ and it contains every place over $p$;
				\item $A$ does not have approximation in $\Sigma$;
				\item $A$ has approximation in $\Sigma\smallsetminus\{\p\}$ for every place $\p$ over $p$.
			\end{itemize}
	\end{theorem}

	Let $n$ be a positive integer, $\ell$ be a prime number and $k=\Q(\zeta_{\ell^n})$. The following result allows us to find an explicit Galois extension of $k$ where the decomposition group of a place of $k$ is its whole Galois group and where places of $k$ dividing the order of this extension have cyclic decomposition groups.

	\begin{lemma} \label{grupo entero p-grupos}
		Let $p$ be a prime congruent to 1 modulo $\ell^n$. Then there exists an odd prime $q$ such that 
			\begin{itemize}
				\item $\Gal(k_\p(\sqrt[\ell^n]{p},\sqrt[\ell]{q})/k_\p)=\Gal(k(\sqrt[\ell^n]{p},\sqrt[\ell]{q})/k)\cong\Z/\ell^n\Z\times\Z/\ell\Z$, for every place $\p\in\Omega_k$ over $p$;
				\item $\Gal(k_\spot{l}(\sqrt[\ell^n]{p},\sqrt[\ell]{q})/k_\spot{l}) \hookrightarrow \Z/\ell^n\Z$, for every place $\spot{l}\in\Omega_k$ over $\ell$.
			\end{itemize}
	\end{lemma}
	\begin{proof}
		By Dirichlet's Theorem there is an odd prime $q$ which is not an $\ell$-th power modulo $p$ and such that
			\begin{itemize}
				\item $\cmod{q}{1}{8}$ if $\ell=2$;
				\item $\cmod{q}{1}{\ell^2}$ if $\ell\neq 2$.
			\end{itemize}
		Set $L=k(\sqrt[\ell^n]{p},\sqrt[\ell]{q})$. Note that $k/\Q$ is unramified at $p$. Then, the polynomial $x^{\ell^n}-p$ is an Eisenstein polynomial at $\p$. So $k(\sqrt[\ell^n]{p})/k$ is a Galois extension with group isomorphic to $\Z/\ell^n\Z$ which is totally ramified at $\p$ and $k_\p(\sqrt[\ell^n]{p})/k_\p$ has Galois group isomorphic to $\Z/\ell^n\Z$. On the other hand, $k(\sqrt[\ell]{q})/k$ is unramified at $\p$ (since $x^\ell-q$ is separable in $\kappa(\p)$) so, $\Gal(k_\p(\sqrt[\ell]{q})/k_\p)\cong \Z/\ell\Z$ if and only if $[\kappa(\P):\kappa(\p)]=\ell$, where $\P$ is a place of $k(\sqrt[\ell]{q})$ over $\p$. We note that $p$ splits completely in $k$ since $\cmod{p}{1}{\ell^n}$, so $\kappa(\p)=\F_p$ and 
			\[ [\kappa(\P):\kappa(\p)]=[\F_p(\sqrt[\ell]{\overline{q}}):\F_p]=\ell, \]
		since $q$ is not an $\ell$-th power modulo $p$. Therefore, $k_\p(\sqrt[\ell]{q})/k_\p$ is a cyclic extension of order $\ell$. Hence, $L_\p:=k_\p(\sqrt[\ell^n]{p},\sqrt[\ell]{q})$ is a ramified extension with
			\[ \Gal(L_\p/k_\p) \cong \Z/\ell^n\Z \times \Z/\ell\Z \cong \Gal(L/k). \]
		Finally, by the properties satisfied by $q$ above we have $q$ is an $\ell$-power in $\Q_\ell$, hence
			\[ \Gal(k_{\spot{l}}(\sqrt[\ell^n]{p},\sqrt[\ell]{q})/k_{\spot{l}}) \hookrightarrow \Z/\ell^n\Z, \]
		for every place $\spot{l}\in\Omega_k$ over $\ell$.
	\end{proof}

	The following result is immediately obtained from previous lemma.
	
	\begin{lemma} \label{buena extension en l-grupos}
		Let $p$ be a prime congruent to 1 modulo $\ell^n$. Then there exists a Galois extension $L/k$ with group $G=\Z/\ell^n\Z\times\Z/\ell\Z$ such that $G_\p=G$ for every place $\p\in\Omega_k$ over $p$ and $G_\frak{l}$ is cyclic for every place $\spot{l}\in\Omega_k$ over $\ell$.
	\end{lemma}

	Thanks to this result we can prove the Theorem \ref{no aproximacion en l-grupos}.

	\begin{proof}[Proof of Theorem \ref{no aproximacion en l-grupos}]
		Let $L/k$ be the extensión given by Lemma \ref{buena extension en l-grupos} and $G:=\Gal(L/k)$. Let $I$ be the augmentation ideal of $\Z/\ell^{n+1}\Z[G]$ equipped with the structure of $k$-group defined by $L/k$. Then, the set $\Sigma_0(k,I)$ contains every place over $p$, it does not contain any place over $\ell$ and, by Lemma \ref{cociente no trivial}, the quotient $\Sh_{\Sigma_0(k,I)}^1(k,I)/\Sh^1(k,I)$ is non trivial. Hence, setting $A:=\hat{I}$ we have, by Proposition \ref{sigma0}, 
			\[ \Sigma_0(k,\hat{A})=\Sigma_0(k,I)\subseteq\Bad_A^r\smallsetminus\Bad_A^d, \]
		$A$ does not have approximation in $\Sigma_0(k,\hat{A})$ and it has approximation in $\Sigma_0(k,\hat{A})\smallsetminus\{\p\}$ for every place $\p$ over $p$ (cf.\ Corollary \ref{no achicar}).
	\end{proof}

	\begin{remark} \label{observaciones}
		When $n=1$, the set $\Sigma_0(k,\hat{A})$ is the set of places whose decomposition group is $\Gal(L/k)$. Thus we have that $\Sigma_0(k,\hat{A})$ is the smallest subset of $\Omega_k$ where $A$ does not have approximation. 
		
		In particular, if $n=1$ and $\ell=2$ we have $k=\Q$ and $\Sigma_0(\Q,I)=\{p,q\}$: indeed, the only primes of $\Q$ whose decomposition group in $L=\Q(\sqrt{p},\sqrt{q})/\Q$ is not cyclic are $p$ and possibly $q$ (cf.\ Lemma \ref{grupo entero p-grupos}). Then $\Sigma_0(\Q,I)\subset\{p,q\}$. Now, by hypothesis $q$ is congruent to 1 modulo 4 and it is not a square modulo $p$. Then by the law of quadratic reciprocity we have that $p$ is not a square modulo $q$ either. Therefore, the decomposition group of $q$ in $L/\Q$ is $\Gal(L/\Q)$. Hence, $\Sigma_0(\Q,I)=\{p,q\}=\Bad^r_A\smallsetminus\Bad^d_A$.
	\end{remark}

%
%

\subsection{Failure of approximation at almost every place of a number field}

	The aim of this subsection is to find a family of counterexamples to the approximation in every number field $k$ and in every place of $k$ not dividing 2. This is summarized in the following result.
	
	\begin{theorem} \label{no aproximacion en cuerpos}
		For every number field $k$ and for every place $\p\in\Omega_k$ which does not divide 2, there exists a finite set of places $\Sigma\subseteq\Omega_k$ and a finite abelian $k$-group $A$ such that
			\begin{itemize}
				\item $\p\in\Sigma\subseteq\Bad_A^r\smallsetminus\Bad_A^d$;
				\item $A$ does not have approximation in $\Sigma$;
				\item $A$ has approximation in $\Sigma\smallsetminus\{\p\}$.
			\end{itemize}
	\end{theorem}

	Let $k$ be a number field and $\p$ be a place of $k$ over a prime $p\in\Q$.

	\begin{lemma} \label{grupo entero cuerpo}
		Suppose that $p$ is distinct from 2. Then, there exists algebraic integers $a,b\in\ent_k$ such that
			\begin{itemize}
				\item $\Gal(k_\p(\sqrt{a},\sqrt{b})/k_\p) = \Gal(k(\sqrt{a},\sqrt{b})/k) \cong\Z/2\Z\times\Z/2\Z$;
				\item $\Gal(k_v(\sqrt{a},\sqrt{b})/k_v) \hookrightarrow \Z/2\Z$, for every place $v\in\Omega_k$ over $2$.
			\end{itemize}
	\end{lemma}
	\begin{proof}
		Let $\Sigma_2\subset\Omega_k$ be the set of places over 2. Set $a_\p\in \ent_\p^*\smallsetminus \ent_\p^{*2}$ such that $\overline{a_\p}$ is not a square in $\kappa(\p)$. Since $\ent_v^{*2}$ is open in $\ent_v^*$ for every $v\in\Sigma_2$, we have that, by the Approximation Theorem on number fields (cf.\ \cite[Chapter II, Theorem 3.4]{NeukirchANT}), there exists an $a\in\ent_k$ such that $a-a_\p\in\p\ent_\p$ and $a\in\ent_v^{*2}$ for each $v\in\Sigma_2$. In particular, $a$ is not a square modulo $\p$. Therefore, $k_v(\sqrt{a})/k_v$ is trivial for every $v\in\Sigma_2$ and $k_\p(\sqrt{a})/k_\p$ is a quadratic extension since $\p$ does not divide 2 and 
			\[ 2 = [\kappa(\p)(\sqrt{\overline{a}}):\kappa(\p)] \leq [k_\p(\sqrt{a}):k_\p] \leq 2. \]
		On the other hand, there exists $b\in k^*$ such that $v_\p(b)=1$. So, the extension $k_\p(\sqrt{b})/k_\p$ is a totally ramified quadratic extension. Hence,
			\[ \Gal(k_\p(\sqrt{a},\sqrt{b})/k_\p) = \Gal(k(\sqrt{a},\sqrt{b})/k) \cong\Z/2\Z\times\Z/2\Z, \]
		and
			\[ \Gal(k_v(\sqrt{a},\sqrt{b})/k_v) = \Gal(k_v(\sqrt{p})/k_v) \hookrightarrow \Z/2\Z, \]
		for every $v\in\Sigma_2$.
	\end{proof}
	
	\begin{remark}
		If $p$ does not divide the discriminant of $k/\Q$, then we can take $b=p$.
	\end{remark}
	
	The following result is immediately obtained from the previous lemma.
	
	\begin{lemma} \label{buena extension}
		Suppose that $p$ is distinct from 2. Then there is a biquadratic extension $L/k$ such that $G_\p=\Gal(L/k)$ and $G_v$ is cyclic for every place $v\in\Omega_k$ over $2$.
	\end{lemma}

	Thanks to this result we can prove Theorem \ref{no aproximacion en cuerpos}. 
	
	\begin{proof}[Proof of Theorem \ref{no aproximacion en cuerpos}]
		Let $L/k$ be the extension given by Lemma \ref{buena extension} and $G:=\Gal(L/k)$. Let $I$ be the augmentation ideal of $\Z/4\Z[G]$ equipped with the structure of abelian $k$-group defined by $L/k$. The proof continues in the same way as Theorem \ref{no aproximacion en l-grupos} replacing $\ell$ by 2. 
	\end{proof}
	
	\begin{remark}
		If we would like to prove the Theorem \ref{no aproximacion en cuerpos} when $p=2$ using the Lemma \ref{cociente no trivial}, then $G$ cannot be asumed abelian in general. Indeed, if $\kappa(\p)=\F_2$ (e.g. when $k/\Q$ is totally ramified at $\p$), then $k_\p^*\cong\Z\times\ent_\p^*$ where $\ent_\p^*$ is a pro-2-group. On the other hand, the reciprocity map in $k_\p$ gives a correspondence between finite abelian extensions of $k_\p$ and finite quotients of $k_\p^*$ (cf. \cite[Théorème 9.13]{HarariCG}). Hence, there is no non cyclic finite abelian extension of $k_\p$ of degree prime to 2. 
	\end{remark}
	
	\begin{remark}
		In Theorems \ref{no aproximacion en l-grupos} and \ref{no aproximacion en cuerpos} we give counterexamples to the approximation for:
			\begin{itemize}
				\item abelian $k$-groups of order an $\ell$-th power for every prime $\ell$ (Theorem \ref{no aproximacion en l-grupos});
				\item and at every place not dividing 2 of an arbitrary number field (Theorem \ref{no aproximacion en cuerpos}).
			\end{itemize}
		This brings up the following question, would it be possible to use the same constructions in order to find a counterexample to the approximation with a completely arbitrary choice of $\ell$, $\p$ and $k$? More precisely, given a prime $\ell$, a number field $k$ and a place $\p\in\Omega_k$ over a prime $p$ distinct from $\ell$, does there exist a non-cyclic Galois extension $L/k$ of degree an $\ell$-th power such that the set $\Sigma_0(k,I)\subseteq\Bad_I^r\smallsetminus\Bad_I^d$ contains $\p$ and $G_\p=\Gal(L/k)$? The answer is no. Indeed, let $\frak{P}$ be a place of $L$ over $\p$. If we suppose that $G_\p=\Gal(L/k)$ then $L_\P/k_\p$ is tamely ramified. Let $T/k_\p$ be the maximal unramified subextension of $L_\P/k_\p$, then the extension $L_\P/T$ is generated by radicals of order divisible by $\ell$ (cf.\ \cite[Chapter II, Proposition 7.7]{NeukirchANT}). This implies that $\mu_\ell\subseteq T$ and thus $\ell$ divides $|\kappa_T|-1$, where $\kappa_T$ is the residue field of $T$. Hence, we cannot expect to improve the results obtained in theorems \ref{no aproximacion en l-grupos} and \ref{no aproximacion en cuerpos}. That is, using the same constructions we cannot find examples with completely arbitrary choice of $\ell$, $\p$ and $k$ since we will inevitably have the restriction $\ell$ divides $|\kappa_T|-1$. 
	\end{remark}

\bibliographystyle{alpha}
\bibliography{article}

\end{document}